\documentclass[a4paper,11pt,oneside]{article}
\usepackage{graphicx}
\usepackage{amscd}
\usepackage{amsmath}
\usepackage{caption}
\usepackage{amsfonts}
\usepackage{amssymb}
\usepackage{amsthm}
\usepackage{mathrsfs}
\usepackage{multicol}
\usepackage{color}
\usepackage[english]{babel}
\usepackage[T1]{fontenc}
\usepackage{textcomp}
\usepackage[utf8]{inputenc}
\usepackage{enumitem}
\usepackage{indentfirst}

\newcommand{\N}{\mathbb N}

\newcommand{\RR}{{{\rm I} \kern -.15em {\rm R} }}

\begin{document}
	\theoremstyle{plain} \newtheorem{thm}{Theorem}[section] \newtheorem{cor}[thm]{Corollary} \newtheorem{lem}[thm]{Lemma} \newtheorem{prop}[thm]{Proposition} \theoremstyle{definition} \newtheorem{defn}{Definition}[section] 

\newtheorem{oss}[thm]{Remark}
	\newtheorem{ex}{Example}[section]
	\newtheorem{lemma}{Lemma}[section]

\title{Consensus for the Hegselmann-Krause model   with\\ time
 variable time delays}
\author{{\sc Elisa Continelli \& Cristina Pignotti}
\\
Dipartimento di Ingegneria e Scienze dell'Informazione e Matematica\\
 Universit\`{a} degli Studi di L'Aquila\\
Via Vetoio, Loc. Coppito, 67010 L'Aquila Italy}

\maketitle

\begin{abstract}
In this paper, we analyze a Hegselmann-Krause opinion formation model with time-variable  time delay and prove that, if the influence function is always positive, then there is exponential convergence to consensus without requiring any smallness assumptions on the time delay function. The analysis is then extended to a model with distributed time delay.
\end{abstract}

\vspace{5 mm}

	\section{Introduction}

\setcounter{equation}{0}
Multiagent systems attracted, in recent years, the attention of many researchers in several scientific disciplines, such as biology \cite{Cama, CS1}, economics \cite{Marsan}, robotics \cite{Bullo, Jad}, control theory \cite{Aydogdu, Borzi, WCB, PRT, Piccoli}, social sciences \cite{Bellomo, Campi, CF}. In particular, we mention the celebrated Hegselmann-Krause opinion formation model \cite{HK} (see \cite{CFT, CFT2} for the related PDE model) and its second-order version, i.e. the  Cucker-Smale model \cite{CS1}  introduced to describe flocking phenomena. A typical feature is the emergence of a collective behavior, namely, under quite general assumptions, solutions to such systems converge to consensus or, in the case of the CS-model, to flocking. 
It is also natural to include in such models time delays, taking into account the times necessary for each agent to receive information from other agents or reaction times.

 In this paper, we deal with an opinion formation model with time-variable time delay. 
Multiagent systems with time delays have already been studied by some authors. Flocking results for the CS-model with delay have been proved  \cite{LW, CH, CL, CP, PT, HM} in different settings, under a smallness assumption on the time delay size. We mention also \cite{DH} for the analysis of a thermomechanical CS-model with delay. 

Concerning the Hegselmann-Krause model  for opinion formation, convergence to consensus results have been proved in presence of small time delays in \cite{CPP, P, H}.
More recently, in \cite{H3}, a consensus result is proved, in the case of a constant time delay, without requiring any upper bound on the time delay. We mention also \cite{Lu} for a consensus result without any smallness assumptions on the time delay size but in the particular case of constant interaction coefficients. Finally,  a flocking result has been recently obtained by \cite{Cartabia} for a CS-model with constant time delay without any restrictions on the time delay size, applying a step by step procedure.  We mention also \cite{PR} for a flocking result without smallness assumption on the time delay related to a CS-model with leadership. Here, we extend the argument of \cite{Cartabia} to the Hegselmann-Krause opinion formation model in the case of a time variable time delay. We then improve previous convergence to consensus results by removing the smallness assumption on the time variable time delay. We are also able to consider a more general influence function, without monotonicity assumptions.   

	Consider a finite set of $N\in\N$ particles, with $N\geq 2 $. Let $x_{i}(t)\in \RR^d$ be the opinion of the $i$-th particle at time $t$. We shall denote with $\lvert\cdot \rvert$ and $\langle\cdot,\cdot\rangle$ the usual norm and scalar product on $\RR^{d}$, respectively. The interactions between the elements of the system are described by the following  Hegselmann-Krause type model with  variable time delays
	\begin{equation}\label{hkp}
	\frac{d}{dt}x_{i}(t)=\underset{j:j\neq i}{\sum}a_{ij}(t)(x_{j}(t-\tau (t))-x_{i}(t)),\quad t>0,	 \forall i=1,\dots,N,
	\end{equation}
	with weights $a_{ij}$ of the form
	\begin{equation}\label{weight}
	a_{ij}(t):=\frac{1}{N-1}\psi( x_{i}(t), x_{j}(t-\tau (t))), \quad\forall t>0,\, \forall i,j=1,\dots,N,
	\end{equation}
	where $\psi:\RR^d\times\RR^d\rightarrow \RR$ is a positive function, and the time delay  $\tau :[0,\infty)\rightarrow[0,\infty)$ is a continuous function satisfying
	\begin{equation}\label{delay}
	0\leq \tau(t)\leq \bar\tau, \quad \forall t\ge 0, 
	\end{equation} for some positive constant $\bar\tau.$ The initial conditions
	\begin{equation}\label{incond}
	x_{i}(s)=x^{0}_{i}(s),\quad \forall s\in [-\bar\tau,0],\,\forall i=1,\dots,N,
	\end{equation}
	are assumed to be continuous functions.

The influence function $\psi$ is assumed to be continuous. Moreover, we assume that it is bounded and we denote
$$K:=\lVert \psi\rVert_{\infty}.$$ 
	For existence results for the above model, we refer e.g. to \cite{Halanay, HL}. Here, we will concentrate on the asymptotic behavior of solutions.

	For each $t\geq-\bar\tau$, we define the diameter $d(\cdot)$ as
 $$d(t):=\max_{i,j=1,\dots,N}\lvert x_{i}(t)-x_{j}(t)\rvert.$$
	\begin{defn}\label{consensus}
		We say that a solution $\{x_{i}\}_{i=1,\dots,N}$ to system \eqref{hkp} converges to \textit{consensus} if $$\lim_{t\to+\infty}d(t)=0.$$
	\end{defn}
We will prove the following convergence to consensus result.

\begin{thm}\label{cons}
 Assume that $\psi:\RR^d\times \RR^d\rightarrow\RR$ is a positive, bounded, continuous function. Moreover, let $x_{i}^{0}:[-\bar\tau,0]\rightarrow\RR^{d}$ be a continuous function,  for any $i=1,\dots,N.$ Then, every solution $\{x_{i}\}_{i=1,\dots,N}$ to \eqref{hkp} under the initial conditions \eqref{incond} satisfies the exponential decay estimate 
\begin{equation}\label{estimate}
		d(t)\leq \left(\max_{i,j=1,\dots,N}\,\,\max_{r,s\in [-\bar\tau,0]}\lvert x_{i}(r)-x_{j}(s)\rvert\right)e^{-\gamma(t-2\bar{\tau})},\quad \forall  t\geq 0,
	\end{equation}
	for a suitable positive constant $\gamma$ independent of $N.$
\end{thm}
We will also consider the continuity type equation obtained as mean-field limit of the particle model when the number $N$ of the agents tends to infinity. Indeed,
since the constants appearing in the exponential consensus estimate for the discrete model are independent of the number of the agents, we can extend the consensus result to the related PDE model.

Moreover, we extend the results obtained for the Hegselmann-Krause model with a pointwise time delay to a model with distributed time delay, namely each agent is influenced by other agents' opinions in a certain time interval (cf. \cite{CP, P}).
In particular, we consider the system
\begin{equation}\label{hkd}
\frac{d}{dt}x_{i}(t)=\frac{1}{h(t)}\underset{j:j\neq i}{\sum}\int_{t-\tau_{2}(t)}^{t-\tau_{1}(t)}\alpha(t-s)a_{ij}(t;s) (x_{j}(s)-x_{i}(t)) ds,\quad t>0,\quad	 \forall i=1,\dots,N,
\end{equation}
where the time delays  $\tau_{1}:[0,\infty)\rightarrow[0,\infty)$, $\tau_{2}:[0,\infty)\rightarrow[0,\infty)$ are continuous functions satisfying \begin{equation}\label{delay_distr}
	0\leq\tau_{1}(t)<\tau_{2}(t)\leq\bar{\tau}, \quad\forall t\geq 0,
\end{equation} 
for some positive constant $\bar{\tau}$.
\\The communication rates $a_{ij}(t;s)$ are of the form
\begin{equation}\label{weightd}
a_{ij}(t;s):=\frac{1}{N-1}\psi(x_{i}(t), x_{j}(s)), \quad\forall t\geq 0,\, \quad\forall i,j=1,\dots,N,
\end{equation}
where $\psi:\RR^d\times\RR^d\rightarrow \RR$ is a positive function. 
\\Moreover, $\alpha:[0,\bar{\tau}]\rightarrow (0,+\infty)$ is a continuous weight function and
\begin{equation}\label{h(t)}
	h(t):=\int_{\tau_{1}(t)}^{\tau_{2}(t)}\alpha(s)ds, \quad\forall t\geq 0.
\end{equation}
Note that, since we assume $\tau_1(t)<\tau_2(t)$ and $\alpha (t)>0,$ $\forall t\geq 0,$ then the function $h(t)$ is always positive.

As before,  the initial conditions
$$
x_{i}(s)=x^{0}_{i}(s),\quad \forall s\in [-\bar\tau,0],\,\forall i=1,\dots,N,
$$
are assumed to be continuous functions.
Moreover, the influence function $\psi$ is assumed to be continuous and bounded, and let us denote
$K:=\lVert \psi\rVert_{\infty}.$

Also in this case, we obtain an exponential consensus estimate without any restrictions on the time delays sizes.
This extends and improves the analysis in \cite{P} where a consensus estimate has been obtained, in the case $\tau_1(t)\equiv 0,$ $\forall t\geq 0,$ subject to a smallness assumption on the time delay size. Moreover, here, as for the pointwise time delay case, we do not require any monotonicity properties on the influence function $\psi$ that is assumed only continuous and bounded.

Even in the distributed case, since the constants in the consensus estimate are independent of the number of the agents, one can extend the consensus theorem to the related PDE model.

The rest of the paper is organized as follows. In Sect. \ref{Prel} we give some preliminary results based on continuity arguments and Gronwall's inequality. In Sec. \ref{consensus_sec} we prove our consensus result for the particle model with pointwise time variable time delay, while in Sect. \ref{PDE} we formulate its extension to the related continuity type equation.
Finally, in Sect. \ref{distributed} we analyze the HK-model with distributed time delay \eqref{hkd}.
 
\section{Preliminaries}\label{Prel}
\setcounter{equation}{0}

\noindent Let $\{x_{i}\}_{i=1,\dots,N}$ be solution to \eqref{hkp} under the initial conditions \eqref{incond}. In this section we present some auxiliary lemmas.  We assume that the hypotheses of Theorem \ref{cons} are satisfied.

The following arguments  generalize and extend  the ones developed in   \cite{Cartabia} in the case of a Cucker-Smale model with constant time delay.
\begin{lem}\label{L1}
	For each $v\in \RR^{d}$ and $T\ge 0,$  we have that 
\begin{equation}\label{scalpr}
		\min_{j=1,\dots,N}\min_{s\in[T-\bar\tau,T]}\langle x_{j}(s),v\rangle\leq \langle x_{i}(t),v\rangle\leq \max_{j=1,\dots,N}\max_{s\in[T-\bar\tau,T]}\langle x_{j}(s),v\rangle,
	\end{equation}for all  $t\geq T-\bar\tau$ and $i=1,\dots,N$. 
\end{lem}
\begin{proof}
	First of all, we note that the inequalities in \eqref{scalpr} are satisfied for every  $t\in [T-\bar\tau,T]$.
	\\Now, let $T\geq 0$. Given a vector $v\in \RR^{d}$, we set $$M_T=\max_{j=1,\dots,N}\max_{s\in[T-\bar\tau,T]}\langle x_{j}(s),v\rangle.$$
	For all $\epsilon >0$, let us define
	$$K^{\epsilon}:=\left\{t>T :\max_{i=1,\dots,N}\langle x_{i}(s),v\rangle < M_T+\epsilon,\,\forall s\in [T,t)\right\}.$$
	By continuity, we have that $K^{\epsilon}\neq\emptyset$. Thus, denoted with $$S^{\epsilon}:=\sup K^{\epsilon},$$
	it holds that $S^{\epsilon}>T$. \\We claim that $S^{\epsilon}=+\infty$. Indeed, suppose by contradiction that $S^{\epsilon}<+\infty$. Note that by definition of $S^{\epsilon}$ it turns out that \begin{equation}\label{max}
		\max_{i=1,\dots,N}\langle x_{i}(t),v\rangle<M_T+\epsilon,\quad \forall t\in (T,S^{\epsilon}),
	\end{equation}
	and \begin{equation}\label{teps}
		\lim_{t\to S^{\epsilon-}}\max_{i=1,\dots,N}\langle x_{i}(t),v\rangle=M_T+\epsilon.
	\end{equation}
	For all $i=1,\dots,N$ and $t\in (T,S^{\epsilon})$, we compute
	$$\frac{d}{dt}\langle x_{i}(t),v\rangle=\frac{1}{N-1}\sum_{j:j\neq i}\psi(x_{i}(t), x_{j}(t-\tau(t)))\langle x_{j}(t-\tau(t))-x_{i}(t),v\rangle.$$
	Notice that, being $t\in (T,S^{\epsilon})$, then  $t-\tau(t)\in (T-\bar\tau, S^{\epsilon})$ and
 \begin{equation}\label{t-tau}
		\langle x_{j}(t-\tau(t)),v\rangle< M_T+\epsilon,\quad \forall j=1, \dots, N.
	\end{equation}
Moreover, \eqref{max} implies that $$\langle x_{i}(t),v\rangle<M_T+\epsilon,$$
	so that $$M_T+\epsilon-\langle x_{i}(t),v\rangle\geq 0.$$Combining this last fact with \eqref{t-tau}, we can write $$\frac{d}{dt}\langle x_{i}(t),v\rangle\leq \frac{1}{N-1}\sum_{j:j\neq i}\psi(x_{i}(t), x_{j}(t-\tau(t)))(M_T+\epsilon-\langle x_{i}(t),v\rangle)$$$$\leq K(M_T+\epsilon-\langle x_{i}(t),v\rangle), \quad\forall t\in (T, S^{\epsilon}).$$
	Then, from Gronwall's inequality we get
 $$\begin{array}{l}
\vspace{0.2cm}\displaystyle{
\langle x_{i}(t),v\rangle\leq e^{-K(t-T)}\langle x_{i}(T),v\rangle+K(M_T+\epsilon)\int_{T}^{t}e^{-K(t-s)}ds}\\
\vspace{0.3cm}\displaystyle{\hspace{1.7 cm}
=e^{-K(t-T)}\langle x_{i}(T),v\rangle+(M_T+\epsilon)e^{-Kt}(e^{Kt}-e^{KT})}\\
\vspace{0.3cm}\displaystyle{\hspace{1.7 cm}
=e^{-K(t-T)}\langle x_{i}(T),v\rangle+(M_T+\epsilon)(1-e^{-K(t-T)})}\\
\vspace{0.3cm}\displaystyle{\hspace{1.7 cm}
\leq e^{-K(t-T)}M_T+M_T+\epsilon -M_Te^{-K(t-T)}-\epsilon e^{-K(t-T)}}\\
\vspace{0.3cm}\displaystyle{\hspace{1.7 cm}
=M_T+\epsilon-\epsilon e^{-K(t-T)}}\\
\displaystyle{\hspace{1.7 cm}
	=M_T+\epsilon-\epsilon e^{-K(S^{\epsilon}-T)},}
\end{array}
$$for all $t\in (T, S^{\epsilon})$.	We have so proved that, $\forall i=1,\dots, N,$
$$\langle x_{i}(t),v\rangle\leq M_T+\epsilon-\epsilon e^{-K(S^{\epsilon}-T)}, \quad \forall t\in (T,S^{\epsilon}).$$
Thus, we get
 \begin{equation}\label{lim}
		\max_{i=1,\dots,N} \langle x_{i}(t),v\rangle\leq M_T+\epsilon-\epsilon e^{-K(S^{\epsilon}-T)}, \quad \forall t\in (T,S^{\epsilon}).
	\end{equation}
	Letting $t\to S^{\epsilon-}$ in \eqref{lim}, from \eqref{teps} we have that $$M_T+\epsilon\leq M_T+\epsilon-\epsilon e^{-K(S^{\epsilon}-T)}<M_T+\epsilon,$$
	which is a contraddiction. Thus, $S^{\epsilon}=+\infty$, which means that $$\max_{i=1,\dots,N}\langle x_{i}(t),v\rangle<M_T+\epsilon, \quad \forall t>T.$$
	From the arbitrariness of $\epsilon$ we can conclude that $$\max_{i=1,\dots,N}\langle x_{i}(t),v\rangle\leq M_T, \quad \forall t>T,$$
	from which $$\langle x_{i}(t),v\rangle\leq M_T, \quad \forall t>T, \,\forall i=1,\dots,N,$$
	which proves the second inequality in \eqref{scalpr}. Now, to prove the other inequality, let $v\in \RR^{d}$ and define $$m_T=\min_{j=1,\dots,N}\min_{s\in[T-\bar{\tau},T]}\langle x_{j}(s),v\rangle.$$
	Then, for all $i=1,\dots,N$ and $t>T$, by applying the second inequality in \eqref{scalpr} to the vector $-v\in\RR^{d}$ we get $$-\langle x_{j}(s),v\rangle=\langle x_{i}(t),-v\rangle\leq \max_{j=1,\dots,N}\max_{s\in[T-\bar{\tau},T]}\langle x_{j}(s),-v\rangle$$$$=-\min_{j=1,\dots,N}\min_{s\in [T-\bar{\tau},T]}\langle x_{j}(s),v\rangle=-m_T,$$
	from which $$\langle x_{j}(s),v\rangle\geq m_T.$$
	Thus, also the first inequality in \eqref{scalpr} is fullfilled.
\end{proof}

We now introduce some notation.

\begin{defn}\label{notation}
	We define
 $$D_{0}=\max_{i,j=1,\dots,N}\,\,\max_{s, t\in [-\bar\tau, 0]}\lvert x_{i}(s)-x_{j}(t)\rvert,$$ 
and in general,  $\forall n\in \mathbb{N}$,
	$$D_{n}:=\max_{i,j=1,\dots,N}\,\,\max_{s,t\in [n\bar\tau-\bar\tau, n\bar\tau]}\lvert x_{i}(s)-x_{j}(t)\rvert.$$
\end{defn}
\noindent Note that  inequality \eqref{estimate} can be written as 
$$d(t)\leq e^{-\gamma(t-2\bar\tau)}D_{0},\quad \forall t\geq 0.$$
Let us denote with $\N_0:=\N\cup\{0\}.$

\begin{lem}
For each $n\in \mathbb{N}_{0}$ and $i,j=1,\dots,N$, we get \begin{equation}
\label{dist}
	\lvert x_{i}(s)-x_{j}(t)\rvert\leq D_{n}, \quad\forall s,t\geq n\bar\tau-\bar\tau.
	\end{equation} 
\end{lem}
\begin{proof}
Fix $n\in\mathbb{N}_{0}$ and $i,j=1,\dots,N$. Given $s,t\geq n\bar{\tau}-\bar{\tau}$, if $\lvert x_{i}(s)-x_{j}(t)\rvert=0$ then of course $D_{n}\geq 0=\lvert x_{i}(s)-x_{j}(t)\rvert$. Thus, we can assume $\lvert x_{i}(s)-x_{j}(t)\rvert>0$ and we set $$v=\frac{x_{i}(s)-x_{j}(t)}{\lvert x_{i}(s)-x_{j}(t)\rvert}.$$
	It turns out that $v$ is a unit vector and, by using \eqref{scalpr} with $T=n\bar\tau$ and the Cauchy-Schwarz inequality, we can write 
	$$\begin{array}{l}
	\vspace{0.3cm}\displaystyle{\lvert x_{i}(s)-x_{j}(t)\rvert=\langle x_{i}(s)-x_{j}(t),v\rangle=\langle x_{i}(s),v\rangle-\langle x_{j}(t),v\rangle}\\
	\vspace{0.3cm}\displaystyle{\leq \max_{l=1,\dots,N}\max_{r\in [n\bar{\tau}-\bar{\tau},n\bar{\tau}]}\langle x_{l}(r),v\rangle-\min_{l=1,\dots,N}\min_{r\in [n\bar{\tau}-\bar{\tau},n\bar{\tau}]}\langle x_{l}(r),v\rangle}\\
	\vspace{0.3cm}\displaystyle{\leq \max_{l,k=1,\dots,N}\max_{r,\sigma\in [n\bar{\tau}-\bar{\tau},n\bar{\tau}]}\langle x_{l}(r)-x_{k}(\sigma),v\rangle}	\\
	\displaystyle{\leq \max_{l,k=1,\dots,N}\max_{r, \sigma\in [n\bar{\tau}-\bar{\tau},n\bar{\tau}]}\lvert x_{l}(r)-x_{k}(\sigma)\rvert\lvert v\rvert=D_{n},}
	\end{array}$$
	which proves \eqref{dist}.
\end{proof}
\begin{oss}
Let us note that from \eqref{dist}, in particular, it follows that
\begin{equation}
\label{dx}
\lvert x_{i}(s)-x_{j}(t)\rvert\leq D_{0}, \quad\forall s,t\geq -\bar{\tau}.
\end{equation}
Moreover, it holds
 \begin{equation}\label{dec}
	D_{n+1}\leq D_{n},\quad \forall n\in \mathbb{N}_{0}.
\end{equation}
\end{oss}

With an analogous argument, one can find a bound on $\vert x_i(t)\vert,$ uniform with respect to $t$ and $i=1,\dots,N.$ Indeed, we have the following lemma.
\begin{lem}
	For every $i=1,\dots,N,$ we have that \begin{equation}\label{boundsol}
	\lvert x_{i}(t)\rvert\leq M^{0},\quad \forall t\geq-\bar{\tau},
	\end{equation}
	where 
$$M^{0}:=\max_{i=1,\dots,N}\,\,\max_{s\in [-\bar\tau, 0]}\lvert x_{i}(s)\rvert.$$
\end{lem}
\begin{proof}
	Given $i=1,\dots,N$ and $t\geq-\bar{\tau}$, if $\lvert x_{i}(t)\rvert =0$ then trivially $M^{0}\geq 0=\lvert x_{i}(t)\rvert $. On the contrary, if $\lvert x_{i}(t)\rvert >0$, we define $$v=\frac{x_{i}(t)}{\lvert x_{i}(t)\rvert},$$
	which is a unit vector for which we can write
	$$\lvert x_{i}(t)\rvert=\langle x_{i}(t),v\rangle. $$
	Then, by applying \eqref{scalpr} for $T=0$ and by using the Cauchy-Schwarz inequality we get $$\lvert x_{i}(t)\rvert\leq \max_{j=1,\dots,N}\max_{s\in [-\bar{\tau},0]}\langle x_{j}(s),v\rangle\leq \max_{j=1,\dots,N}\max_{s\in [-\bar{\tau},0]}\lvert x_{j}(s)\rvert\lvert v\rvert$$$$=\max_{j=1,\dots,N}\max_{s\in [-\bar{\tau},0]}\lvert x_{j}(s)\rvert=M^{0},$$
	which proves \eqref{boundsol}.
\end{proof}
\begin{oss}
From the  estimate \eqref{boundsol}, since the influence function $\psi$ is continuous, we deduce that 
\begin{equation}\label{stima_psi}
\psi (x_i(t), x_j(t-\tau(t)))\ge \psi_0:=\min_{\vert y\vert, \vert z\vert \le M^{0}}\psi(y,z)>0,
\end{equation}
for all $t\ge 0,$ for all $i,j=1,\dots, N.$
\end{oss}

\begin{lem}
	For all $i,j=1,\dots,N$,  unit vector $v\in \RR^{d}$ and $n\in\mathbb{N}_{0}$ we have that 
	\begin{equation}\label{4}
		\langle x_{i}(t)-x_{j}(t),v\rangle\leq e^{-K(t-t_{0})}\langle x_{i}(t_{0})-x_{j}(t_{0}),v\rangle+(1-e^{-K(t-t_{0})})D_{n},
	\end{equation}
	for all $t\geq t_{0}\geq n\bar{\tau}$. \\Moreover, for all $n\in \mathbb{N}_{0}$, we get
	\begin{equation}\label{n+1}
		D_{n+1}\leq e^{-K\bar{\tau}}d(n\bar{\tau})+(1-e^{-K\bar{\tau}})D_{n}.
	\end{equation}
\end{lem}
\begin{proof}
	Fix $n\in\mathbb{N}_{0}$ and $v\in\RR^{d}$ such that $\lvert v\rvert=1$. We set $$M_n=\max_{i=1,\dots,N}\max_{t\in [n\bar{\tau}-\bar{\tau},n\bar{\tau}]}\langle x_{i}(t),v\rangle,$$$$m_n=\min_{i=1,\dots,N}\min_{t\in [n\bar{\tau}-\bar{\tau},n\bar{\tau}]}\langle x_{i}(t),v\rangle.$$
	Then, it is easy to see that $M_n-m_n\leq D_{n}$. Now, for all $i=1,\dots,N$ and $t\geq t_{0}\geq n\bar{\tau}$ we have that 

$$
\begin{array}{l}
\displaystyle{
\frac{d}{dt}\langle x_{i}(t),v\rangle=\sum_{j:j\neq i}a_{ij}(t)\langle x_{j}(t-\tau(t))-x_{i}(t),v\rangle}\\
\displaystyle{\hspace{2 cm}=\frac{1}{N-1}\sum_{j:j\neq i}\psi( x_{i}(t), x_{j}(t-\tau(t)))(\langle x_{j}(t-\tau(t)),v\rangle-\langle x_{i}(t),v\rangle)}\\
\displaystyle{\hspace{2 cm}\leq \frac{1}{N-1}\sum_{j:j\neq i}\psi(x_{i}(t), x_{j}(t-\tau(t)))(M_n-\langle x_{i}(t),v\rangle).}
\end{array}
$$

	Note that, being $t\geq n\bar{\tau}$, $\langle x_{i}(t),v\rangle\leq M_n$ from \eqref{scalpr}. Therefore, we have that $M_n-\langle x_{i}(t),v\rangle\geq 0$ and we can write $$\frac{d}{dt}\langle x_{i}(t),v\rangle\leq\frac{1}{N-1}K\sum_{j:j\neq i}(M_n-\langle x_{i}(t),v\rangle)=K(M_n-\langle x_{i}(t),v\rangle).
$$
	Thus, from the Gronwall's inequality it comes that $$\langle x_{i}(t),v\rangle\leq e^{-K(t-t_{0})}\langle x_{i}(t_{0}),v\rangle+\int_{t_{0}}^{t}KM_ne^{-K(t-t_{0})+K(s-t_{0})}ds$$$$\hspace{1.3cm}=e^{-K(t-t_{0})}\langle x_{i}(t_{0}),v\rangle+e^{-K(t-t_{0})}M_n(e^{K(t-t_{0})}-1),$$
	that is \begin{equation}\label{2}
		\langle x_{i}(t),v\rangle\leq e^{-K(t-t_{0})}\langle x_{i}(t_{0}),v\rangle+(1-e^{-K(t-t_{0})})M_n.
	\end{equation}
	On the other hand, for all $i=1,\dots,N$ and $t\geq t_{0}\geq n\bar{\tau}$ it holds that
 $$
\begin{array}{l}
\displaystyle{
\frac{d}{dt}\langle x_{i}(t),v\rangle=\frac{1}{N-1}\sum_{i:j\neq i}\psi(x_{i}(t),x_{j}(t-\tau(t)))(\langle x_{j}(t-\tau(t)),v\rangle-\langle x_{i}(t),v\rangle)}\\
\displaystyle{
\hspace{2 cm}\geq \frac{1}{N-1}\sum_{j:j\neq i}\psi(x_{i}(t),x_{j}(t-\tau(t)))(m_n-\langle x_{i}(t),v\rangle).}
\end{array}
$$
	Note that from \eqref{scalpr} $\langle x_{i}(t),v\rangle\geq m_n$ since $t\geq n\bar{\tau}$. Thus, $m_n-\langle x_{i}(t),v\rangle\leq0$ and, by recalling that $\psi$ is bounded, we get $$\frac{d}{dt}\langle x_{i}(t),v\rangle\geq K(m_n-\langle x_{i}(t),v\rangle).$$
	Hence, by using the Gronwall's inequality it turns out that
	\begin{equation}\label{3}
		\langle x_{i}(t),v\rangle\geq e^{-K(t-t_{0})}\langle x_{i}(t_{0}),v\rangle+(1-e^{-K(t-t_{0})})m_n.
	\end{equation}
	Therefore, for all $i,j=1,\dots,N$ and $t\geq t_{0}\geq n\bar{\tau}$, by using \eqref{2} and \eqref{3} and by recalling that $M_n-m_n\leq D_{n},$ we finally get 
$$
\begin{array}{l}
\vspace{0.3cm}\displaystyle{\langle x_{i}(t)-x_{j}(t),v\rangle=\langle x_{i}(t),v\rangle-\langle x_{j}(t),v\rangle}\\
\vspace{0.3cm}\displaystyle{\hspace{2.8 cm}\leq e^{-K(t-t_{0})}\langle x_{i}(t_{0}),v\rangle+(1-e^{-K(t-t_{0})})M_n}\\
\vspace{0.3cm}\displaystyle{\hspace{2.8 cm}-e^{-K(t-t_{0})}\langle x_{j}(t_{0}),v\rangle-(1-e^{-K(t-t_{0})})m_n}\\
\vspace{0.3cm}\displaystyle{\hspace{2.8 cm}=e^{-K(t-t_{0})}\langle x_{i}(t_{0})-x_{j}(t_{0}),v\rangle+(1-e^{-K(t-t_{0})})(M_n-m_n)}\\
\displaystyle{\hspace{2.8 cm}\leq e^{-K(t-t_{0})}\langle x_{i}(t_{0})-x_{j}(t_{0}),v\rangle+(1-e^{-K(t-t_{0})})D_{n},}
\end{array}
$$
	i.e. \eqref{4} holds true.
	\\Now we prove \eqref{n+1}. Given $n\in\mathbb{N}_{0}$, let $i,j=1,\dots,N$ and $s,t\in [n\bar{\tau},n\bar{\tau}+\bar{\tau}]$ be such that $D_{n+1}=\lvert x_{i}(s)-x_{j}(t)\rvert$. Note that, if $\lvert x_{i}(s)-x_{j}(t)\rvert=0$, then obviously 
$$0=D_{n+1}\leq e^{-K\bar{\tau}}d(n\bar{\tau})+(1-e^{-K\bar{\tau}})D_{n}.$$ 
So we can assume $\lvert x_{i}(s)-x_{j}(t)\rvert>0$. Let us define the unit vector $$v=\frac{x_{i}(s)-x_{j}(t)}{\lvert x_{i}(s)-x_{j}(t)\rvert}.$$
	Hence, we can write $$D_{n+1}=\langle x_{i}(s)-x_{j}(t),v\rangle=\langle x_{i}(s),v\rangle-\langle x_{j}(t),v\rangle.$$
	Now, by using \eqref{2} with $t_{0}=n\bar{\tau}$, we have that$$\langle x_{i}(s),v\rangle\leq e^{-K(s-n\bar{\tau})}\langle x_{i}(n\bar{\tau}),v\rangle+(1-e^{-K(s-n\bar{\tau})})M_n$$$$\hspace{0.3 cm}=e^{-K(s-n\bar{\tau})}(\langle x_{i}(n\bar{\tau}),v\rangle-M_n)+M_n.$$
	Thus, since $s\leq n\bar{\tau}+\bar{\tau}$ and $\langle x_{i}(n\bar{\tau}),v\rangle-M_n\leq 0$ from \eqref{scalpr}, we get \begin{equation}\label{5}
		\begin{split}
		&\langle x_{i}(s),v\rangle\leq e^{-K\bar{\tau}}(\langle x_{i}(n\bar{\tau}),v\rangle-M_n)+M_n\hspace{1 cm}\\&\hspace{1.5 cm}\leq e^{-K\bar{\tau}}\langle x_{i}(n\bar{\tau}),v\rangle+(1-e^{-K\bar{\tau}})M_n.
		\end{split}
	\end{equation}	
	Similarly, by taking into account of \eqref{scalpr} and \eqref{3}, we have that 
	\begin{equation}\label{6}
		\langle x_{j}(t),v\rangle\geq e^{-K\bar{\tau}}\langle x_{j}(n\bar{\tau}),v\rangle+(1-e^{-K\bar{\tau}})m_n.
	\end{equation}
	Therefore, combining \eqref{5} and \eqref{6}, we can write$$D_{n+1}\leq e^{-K\bar{\tau}}\langle x_{i}(n\bar{\tau}),v\rangle+(1-e^{-K\bar{\tau}})M_n-e^{-K\bar{\tau}}\langle x_{j}(n\bar{\tau}),v\rangle-(1-e^{-K\bar{\tau}})m_n$$$$=e^{-K\bar{\tau}}\langle x_{i}(n\bar{\tau})- x_{j}(n\bar{\tau}),v\rangle+(1-e^{-K\bar{\tau}})(M_n-m_n).$$Then, by recalling that $M_n-m_n\leq D_{n}$ and by using the Cauchy-Schwarz inequality, we can conclude that
	$$D_{n+1}\leq e^{-K\bar{\tau}}\lvert x_{i}(n\bar{\tau})-x_{j}(n\bar{\tau})\rvert\lvert v\rvert+(1-e^{-K\bar{\tau}})D_{n}$$$$\leq e^{-K\bar{\tau}}d(n\bar{\tau})+(1-e^{-K\bar{\tau}})D_{n}.$$
\end{proof}
\begin{lem}
	There exists a constant $C\in (0,1),$ independent of $N\in\N,$ such that
	\begin{equation}\label{n-2}
		d(n\bar{\tau})\leq C D_{n-2},
	\end{equation}for all $n\geq 2$.
\end{lem}
\begin{proof}
	Trivially, if $d(n\bar{\tau})=0$, then of course inequality \eqref{n-2} holds for any constant $C\in (0,1)$. So, suppose $d(n\bar{\tau})>0$. Let $i,j=1,\dots,N$ be such that $d(n\bar{\tau})=\lvert x_{i}(n\bar{\tau})-x_{j}(n\bar{\tau})\rvert$. We set $$v=\frac{x_{i}(n\bar{\tau})-x_{j}(n\bar{\tau})}{\lvert x_{i}(n\bar{\tau})-x_{j}(n\bar{\tau})\rvert}.$$
	Then, $v$ is a unit vector for which we can write $$d(n\bar{\tau})=\langle x_{i}(n\bar{\tau})-x_{j}(n\bar{\tau}),v\rangle.$$
	Let us define $$M_{n-1}=\max_{l=1,\dots,N}\max_{s\in [n\bar{\tau}-2\bar{\tau},n\bar{\tau}-\bar{\tau}]}\langle x_{l}(s),v\rangle,$$
	$$m_{n-1}=\min_{l=1,\dots,N}\min_{s\in [n\bar{\tau}-2\bar{\tau},n\bar{\tau}-\bar{\tau}]}\langle x_{l}(s),v\rangle.$$
	Then $M_{n-1}-m_{n-1}\leq D_{n-1}$. Now, we distinguish two different situations.
	\par\textit{Case I.} Assume that there exists $t_{0}\in [n\bar{\tau}-2\bar{\tau},n\bar{\tau}]$ such that 
	$$\langle x_{i}(t_{0})-x_{j}(t_{0}),v\rangle<0.$$
	Then from \eqref{4} with $n\bar{\tau}\geq t_{0}\geq n\bar{\tau}-2\bar{\tau}$ we have that $$d(n\bar{\tau})\leq e^{-K(n\bar{\tau}-t_{0})}\langle x_{i}(t_{0})-x_{j}(t_{0}),v\rangle+(1-e^{-K(n\bar{\tau}-t_{0})})D_{n-2}$$$$\leq (1-e^{-K(n\bar{\tau}-t_{0})})D_{n-2}\leq (1-e^{-2K\bar{\tau}})D_{n-2}.$$
	\par\textit{Case II.} Suppose that \begin{equation}\label{pos}
		\langle x_{i}(t)-x_{j}(t),v\rangle\geq 0,\quad \forall t\in [n\bar{\tau}-2\bar{\tau},n\bar{\tau}].
	\end{equation}
	Then, for every  $t\in [n\bar{\tau}-\bar{\tau},n\bar{\tau}]$ we have that 
$$
\begin{array}{l}
\displaystyle{
\frac{d}{dt}\langle x_{i}(t)-x_{j}(t),v\rangle=\frac{1}{N-1}\sum_{l:l\neq i}\psi(x_{i}(t),x_{l}(t-\tau(t)))\langle x_{l}(t-\tau(t))-x_{i}(t),v\rangle}\\
\displaystyle{\hspace{1.1cm}-\frac{1}{N-1}\sum_{l:l\neq j}\psi(x_{i}(t),x_{l}(t-\tau(t)))\langle x_{l}(t-\tau(t))-x_{j}(t),v\rangle}\\
\displaystyle{\hspace{0.6cm}=\frac{1}{N-1}\sum_{l:l\neq i}\psi(x_{i}(t),x_{l}(t-\tau(t)))(\langle x_{l}(t-\tau(t)),v\rangle-M_{n-1}+M_{n-1}-\langle x_{i}(t),v\rangle)}\\
\displaystyle{\hspace{1.1cm}+\frac{1}{N-1}\sum_{l:l\neq j}\psi(x_{i}(t),x_{l}(t-\tau(t)))(\langle x_{j}(t),v\rangle-m_{n-1}+m_{n-1}-\langle x_{l}(t-\tau(t)),v\rangle)}\\
\displaystyle{\hspace{5.5 cm}:=S_1+S_2.}
\end{array}
$$
	Now, being $t\in [n\bar{\tau}-\bar{\tau},n\bar{\tau}]$, it holds that $t-\tau(t)\in [n\bar{\tau}-2\bar{\tau},n\bar{\tau}]$. Therefore, both $t,t-\tau(t)\geq n\bar{\tau}-2\bar{\tau}$ and from \eqref{scalpr} we have that
	\begin{equation}\label{7}
		m_{n-1}\leq\langle x_{k}(t),v\rangle\leq M_{n-1}, \quad m_{n-1}\leq\langle x_{k}(t-\tau(t)),v\rangle\leq M_{n-1},\quad \forall k=1,\dots,N.
	\end{equation}
	Therefore, using \eqref{boundsol}, we get
	$$
\begin{array}{l}
\displaystyle{
S_1=\frac{1}{N-1}\sum_{l:l\neq i}\psi(x_{i}(t),x_{l}(t-\tau(t)))(\langle x_{l}(t-\tau(t)),v\rangle-M_{n-1})}\\
\displaystyle{\hspace{0.7 cm}+\frac{1}{N-1}\sum_{l:l\neq i}\psi(x_{i}(t),x_{l}(t-\tau(t)))(M_{n-1}-\langle x_{i}(t),v\rangle)}\\
\displaystyle{\hspace{0.4 cm}\leq \frac{1}{N-1}\psi_{0}\sum_{l:l\neq i}(\langle x_{l}(t-\tau(t)),v\rangle-M_{n-1})+K(M_{n-1}-\langle x_{i}(t),v\rangle)},
\end{array}
$$
	and	
	$$
\begin{array}{l}
\displaystyle{
S_2=\frac{1}{N-1}\sum_{l:l\neq j}\psi(x_{i}(t),x_{l}(t-\tau(t)))(\langle x_{j}(t),v\rangle-m_{n-1})}\\
\displaystyle{\hspace{0.7 cm}+\frac{1}{N-1}\sum_{l:l\neq j}\psi(x_{i}(t),x_{l}(t-\tau(t)))(m_{n-1}-\langle x_{l}(t-\tau(t)),v\rangle)}\\
	\displaystyle{\hspace{0.4 cm}\leq K(\langle x_{j}(t),v\rangle-m_{n-1})+\frac{1}{N-1}\psi_{0}\sum_{l:l\neq j}(m_{n-1}-\langle x_{l}(t-\tau(t)),v\rangle).}
\end{array}
$$
	Combining this last fact with \eqref{7} it comes that 
	$$\begin{array}{l}
	\vspace{0.2cm}\displaystyle{\frac{d}{dt}\langle x_{i}(t)-x_{j}(t),v\rangle\leq K(M_{n-1}-m_{n-1}-\langle x_{i}(t)-x_{j}(t),v\rangle)}\\
	\vspace{0.2cm}\displaystyle{\hspace{1.8 cm}+\frac{1}{N-1}\psi_{0}\sum_{l:l\neq i,j}(\langle x_{l}(t-\tau(t)),v\rangle-M_{n-1}+m_{n-1}-\langle x_{l}(t-\tau(t)),v\rangle)}\\
	\vspace{0.2cm}\displaystyle{\hspace{1.8 cm}+\frac{1}{N-1}\psi_{0}(\langle x_{j}(t-\tau(t)),v\rangle-M_{n-1}+m_{n-1}-\langle x_{i}(t-\tau(t)),v\rangle)}\\
	\vspace{0.2cm}\displaystyle{\hspace{1.5 cm}=K(M_{n-1}-m_{n-1})-K\langle x_{i}(t)-x_{j}(t),v\rangle+\frac{N-2}{N-1}\psi_{0}(-M_{n-1}+m_{n-1})}\\
	\vspace{0.2cm}\displaystyle{\hspace{1.8 cm}+\frac{1}{N-1}\psi_{0}(\langle x_{j}(t-\tau(t)),v\rangle-M_{n-1}+m_{n-1}-\langle x_{i}(t-\tau(t)),v\rangle)}.
\end{array}
$$
Therefore, since from \eqref{pos} $\langle x_{i}(t-\tau(t))-x_{j}(t-\tau(t)),v\rangle\geq 0$, we get
$$\begin{array}{l}
	\vspace{0.2cm}\displaystyle{
\frac{d}{dt}\langle x_{i}(t)-x_{j}(t),v\rangle\leq K(M_{n-1}-m_{n-1})-K\langle x_{i}(t)-x_{j}(t),v\rangle
}\\
	\vspace{0.4 cm}\displaystyle{\hspace{1.8 cm}
+\frac{N-2}{N-1}\psi_{0}(-M_{n-1}+m_{n-1})+\frac{1}{N-1}\psi_{0}(-M_{n-1}+m_{n-1})
}\\
\vspace{0.4 cm}\displaystyle{\hspace{1.8 cm}
-\frac{1}{N-1}\psi_{0}\langle x_{i}(t-\tau(t))-x_{j}(t-\tau(t)),v\rangle
}\\
	\vspace{0.3 cm}\displaystyle{\hspace{1.5 cm}\leq K(M_{n-1}-m_{n-1})-K\langle x_{i}(t)-x_{j}(t),v\rangle+\psi_{0}(-M_{n-1}+m_{n-1})
}\\
	\displaystyle{\hspace{1.5 cm}=\left(K-\psi_{0}\right)(M_{n-1}-m_{n-1})-K\langle x_{i}(t)-x_{j}(t),v\rangle.}
	\end{array}$$
	Hence, from Gronwall's inequality it comes that $$\langle x_{i}(t)-x_{j}(t),v\rangle \leq e^{-K(t-n\bar{\tau}+\bar{\tau})}\langle x_{i}(n\bar{\tau}-\bar{\tau})-x_{j}(n\bar{\tau}-\bar{\tau}),v\rangle$$$$+\left(K-\psi_{0}\right)(M_{n-1}-m_{n-1})\int_{n\bar{\tau}-\bar{\tau}}^{t}e^{-K(t-s)}ds,$$
	for all $t\in [n\bar{\tau}-\bar{\tau},n\bar{\tau}]$. In particular, for $t=n\bar{\tau}$ it comes that 
	$$\begin{array}{l}
	\vspace{0.2cm}\displaystyle{d(n\bar{\tau})\leq e^{-K\bar{\tau}}\langle x_{i}(n\bar{\tau}-\bar{\tau})-x_{j}(n\bar{\tau}-\bar{\tau}),v\rangle+\frac{K-\psi_{0}}{K}(M_{n-1}-m_{n-1})(1-e^{-K\bar{\tau}})}\\
	\vspace{0.2cm}\displaystyle{\hspace{1.3cm}\leq e^{-K\bar{\tau}}\lvert x_{i}(n\bar{\tau}-\bar{\tau})-x_{j}(n\bar{\tau}-\bar{\tau})\rvert\lvert v\rvert+\frac{K-\psi_{0}}{K}(M_{n-1}-m_{n-1})(1-e^{-K\bar{\tau}})}\\
	\displaystyle{\hspace{1.3 cm}\leq e^{-K\bar{\tau}} d(n\bar{\tau}-\bar{\tau})+\frac{K-\psi_{0}}{K}(M_{n-1}-m_{n-1})(1-e^{-K\bar{\tau}})}
	\end{array}$$
	Then, by recalling that $M_{n-1}-m_{n-1}\leq D_{n-1}$ we get
	$$d(n\bar{\tau})\leq e^{-K\bar{\tau}} d(n\bar{\tau}-\bar{\tau})+\frac{K-\psi_{0}}{K}D_{n-1}(1-e^{-K\bar{\tau}})$$$$\hspace{1.4cm}\leq e^{-K\bar{\tau}}d(n\bar{\tau}-\bar{\tau})+\frac{K-\psi_{0}}{K}D_{n-1}(1-e^{-K\bar{\tau}}).$$
	Finally, by using \eqref{dist} and \eqref{dec} we have that that \begin{equation}\label{9}
		\begin{split}
		d(n\bar{\tau})&\leq e^{-K\bar{\tau}}D_{n}+\frac{K-\psi_{0}}{K}D_{n-1}(1-e^{-K\bar{\tau}})\\&\leq e^{-K\bar{\tau}}D_{n-2}+\frac{K-\psi_{0}}{K}D_{n-2}(1-e^{-K\bar{\tau}})\\&=\left[1-\,\frac{\psi_{0}}{K}(1-e^{-K\bar{\tau}})\right]D_{n-2}.
		\end{split}
	\end{equation}
	Now, we set 
\begin{equation}\label{C}
C=\max\left\{1-e^{-2K\bar{\tau}},1-\,\frac{\psi_{0}}{K}(1-e^{-K\bar{\tau}})\right\}\in(0,1).
\end{equation}
	Then, taking into account \eqref{9}, we can conclude that $C$ is the constant for which inequality \eqref{n-2} holds.
\end{proof}
\section{Proof of Theorem \ref{cons}}\label{consensus_sec}

\setcounter{equation}{0}

\begin{proof}
	Let $(x_{i})_{i=1,\dots,N}$ be solution to \eqref{hkp}, \eqref{incond}. We claim that 
	\begin{equation}\label{10}
		D_{n+1}\leq \tilde{C}D_{n-2},\quad \forall n\geq2,
	\end{equation} 
	for some constant $\tilde{C}\in (0,1)$. Indeed, given $n\geq 2$, from \eqref{dec}, \eqref{n+1} and \eqref{n-2} we have that $$\begin{array}{l}
	\vspace{0.3cm}\displaystyle{D_{n+1}\leq e^{-K\bar{\tau}}d(n\bar{\tau})+(1-e^{-K\bar{\tau}})D_{n}}\\
	\vspace{0.3cm}\displaystyle{\hspace{1cm}\leq e^{-K\bar{\tau}}CD_{n-2}+(1-e^{-K\bar{\tau}})D_{n}}\\
	\vspace{0.3cm}\displaystyle{\hspace{1cm}\leq e^{-K\bar{\tau}}CD_{n-2}+(1-e^{-K\bar{\tau}})D_{n-2}}\\
	\displaystyle{\hspace{1cm}\leq(1-e^{-K\bar{\tau}}(1-C)) D_{n-2},}
	\end{array}$$
where the constant $C$ is defined in \eqref{C}.
	So, setting $$\tilde{C}=1-e^{-K\bar{\tau}}(1-C),$$
	we can conclude that $\tilde{C}\in (0,1)$ is the constant for which \eqref{10} holds true.
	\\This implies that \begin{equation}\label{11}
		D_{3n}\leq \tilde{C}^{n}D_{0},\quad \forall n\geq 1.
	\end{equation}
	Indeed, by induction, if $n=1$ we know from \eqref{10} that $$D_{3}\leq \tilde{C}D_{0}.$$
	So, assume that \eqref{11} holds for $n>1$ and we prove it for $n+1$. By using again \eqref{10} and from the induction hypothesis it comes that $$D_{3(n+1)}\leq\tilde{C}D_{3n}\leq \tilde{C}\tilde{C}^{n}D_{0}=\tilde{C}^{n+1}D_{0},$$
	i.e. \eqref{11} is fulfilled.
	\\Notice that, from \eqref{11}, we have that 
$$\displaystyle{D_{3n}\leq\left  (\sqrt[3]{\tilde{C}}\right )^{3n}D_{0}=e^{3n\ln\left(\sqrt[3]{\tilde{C}}\right )}D_{0}=e^{3n\bar{\tau}\ln\left(\tilde{C}\right)\frac{1}{3\bar{\tau}}}D_{0}, \quad \forall n\in\N.}$$
	Therefore, if we set $$\gamma=\frac{1}{3\bar{\tau}}\ln\left(\frac{1}{\tilde{C}}\right),$$
	we can write \begin{equation}\label{12}
		D_{3n}\leq e^{-3n\gamma\bar{\tau}}D_{0},\quad \forall n\in\N_0.
	\end{equation}
Now, fix $i,j=1,\dots,N$ and $t\ge 0.$ Then, $t\in [3n\bar{\tau}-\bar\tau, 3n\bar{\tau}+2\bar\tau]$, for some $n\in \mathbb{N}_0$. Therefore,  by using \eqref{dist} and \eqref{12}, it turns out that 
	$$\lvert x_{i}(t)-x_{j}(t)\rvert\leq D_{3n}\leq e^{-3n\gamma\bar{\tau}}D_{0}.$$
	Thus, being $t\leq 3n\bar{\tau}+2\bar\tau$, we get
$$\lvert x_{i}(t)-x_{j}(t)\rvert\leq e^{-\gamma t}\, e^{2\gamma\bar\tau}D_{0}.$$
Then,
$$d(t)\leq e^{-\gamma (t-2\bar\tau)}D_0,\quad \forall t\ge 0$$
	and \eqref{estimate} is proved. 
\end{proof}
\begin{oss}\label{caso-partic}
	Let us note that Theorem \ref{cons}  holds true, in particular,  for weights $a_{ij}$ of the form
	$$a_{ij}(t):=\frac{1}{N-1}\tilde\psi( \lvert x_{i}(t)-x_{j}(t-\tau (t))\rvert), \quad\forall t>0,\, \forall i,j=1,\dots,N,$$
	where $\tilde\psi:\RR\rightarrow\RR$ is a positive, bounded and continuous function. In this case, we can estimate from below $\tilde\psi$ in terms of the distance between the initial opinions, namely
 $$\tilde\psi(\lvert x_{i}(t)-x_{j}(t-\tau (t))\rvert)\geq \tilde\psi_{0}:=\min_{s\in[0,D_{0}]}\tilde\psi(s)>0,\quad \forall t\geq 0,$$
where $D_0$ is as in Definition \ref{notation}.
	Then, the proof of Theorem \ref{cons} follows with the same arguments we have employed  in the general case of weights of the type \eqref{weight}.
\end{oss}
\section{The continuum HK-model with pointwise time delay}\label{PDE} 

\setcounter{equation}{0}

In this section, we consider the continuum model obtained as mean-field limit of the particle system
when $N\rightarrow \infty.$ Let $\mathcal{M}(\RR^{d})$ be the set of probability measures on the space $\RR^{d}$. Then, the continuum model associated to the particle system \eqref{hkp} is given by
\begin{equation}\label{kinetic}
\begin{array}{l}
\displaystyle{\partial_t \mu_t+ \mbox{\rm div\,} (F[\mu_{t-\tau(t)}]\mu_t)=0, \quad t>0,} 
\\
\displaystyle{\mu_s=g_s, \quad x\in \RR^d, \quad  s\in[-\bar\tau,0],}
\end{array}
\end{equation}
where the velocity field $F$ is defined as 
\begin{equation}
\label{F}
F[\mu_{t-\tau(t)}](x)=\int_{\RR^d} \psi (x, y)(y-x)\,d\mu_{t-\tau(t)}(y),
\end{equation}
and $g_s \in \mathcal{C}([-\bar{\tau},0];\mathcal{M}(\RR^d))$. For  the continuum model, as in \cite{CPP}, we  assume that the delay function $\tau(\cdot)$ is bounded from below, namely there exists a strictly positive constant $\tau^*>0$ such that
\[
\tau(t)\geq \tau^*, \quad \forall \,t\geq 0.
\]
Moreover, we assume that the potential $\psi(\cdot,\cdot)$ in \eqref{F} is also Lipschitz continuous, namely for any $(x,y),(x',y')\in\RR^{2d}$ there exists $L>0$ such that
$$
|\psi(x,y)-\psi(x', y')|\leq L(|y-y'|+|x-x'|).
$$
\begin{defn}
Let $T>0$. We say that $\mu_t\in \mathcal{C}([0,T);\mathcal{M}(\RR^d))$ is a measure-valued solution to \eqref{kinetic} on the time interval $[0,T)$ if for all $\varphi\in \mathcal{C}^\infty_c(\RR^d\times [0,T))$ we have:
\begin{equation}
\label{weak_solution}
\int_0^T \int_{\RR^d} \left( \partial_t\varphi+F[\mu_{t-\tau(t)}](x) \cdot \nabla_x\varphi \right) d\mu_t(x) dt +\int_{\RR^d} \varphi (x,0)dg_0(x) =0.
\end{equation}
\end{defn}
Before stating the consensus result for solutions to model \eqref{kinetic}, we first recall some
basic tools on probability spaces and measures.
\begin{defn}
Let $\mu,\nu\in \mathcal{M}(\RR^d)$ be two probability measures on $\RR^d$. We define the 1-Wasserstein distance between $\mu$ and $\nu$ as
$$
d_1(\mu,\nu):=\inf_{\pi\in \Pi(\mu,\nu)} \int_{\RR^d\times \RR^d} |x-y|d\pi(x,y),
$$
where $\Pi (\mu,\nu)$ is the space of all couplings for $\mu$ and $\nu$, namely all those probability measures on $\RR^{2d}$ having as marginals $\mu$ and $\nu$:
$$
\int_{\RR^d\times\RR^d} \varphi (x)d\pi(x,y)=\int_{\RR^d} \varphi(x)d\mu(x), \quad \int_{\RR^d\times\RR^d} \varphi(y)d\pi(x,y)=\int_{\RR^d}\varphi(y)d\nu(y),
$$
for all $\varphi \in \mathcal{C}_b(\RR^d)$.
\end{defn}
Let us introduce the space $\mathcal{P}_1$ of all probability measures with finite first-order moment. It is  well-known that $(\mathcal{P}_1(\RR^d),d_1(\cdot,\cdot))$ is a complete metric space.

Now, we define the position diameter for a compactly supported measure $g \in {\mathcal{P}}_1(\RR^d)$ as follows:
$$d_X[g]:=\mbox{diam}(supp\ g).$$

Since the consensus result for the particle model \eqref{hkp} holds without any upper bounds on the time delay $\tau(\cdot)$, one can improve the consensus theorem for the PDE model \eqref{kinetic} obtained in \cite{CPP} removing the smallness assumption on the time delay $\tau(t).$
We omit the proof since, once we have the result for the particle system \eqref{hkp}, the consensus estimate for the continuum model is obtained with arguments analogous to the ones in \cite{CPP} (see also \cite{PPpreprint}).

\begin{thm}
Let $\mu_t\in \mathcal{C}([0,T];\mathcal{P}_1(\RR^d))$ be a measure-valued solution to \eqref{kinetic} with compactly supported initial datum $g_s\in \mathcal{C}([-\bar\tau,0];\mathcal{P}_1(\RR^d))$ and let $F$ as in \eqref{F}.
Then, there exists a constant $C>0$  such that 
$$
d_X(\mu_t)\leq \left( \max_{s\in[-\bar\tau,0]} d_X(g_s) \right) e^{-Ct},
$$
for all $t\ge 0.$
\end{thm}

\section{Distributed time delay}\label{distributed}

\setcounter{equation}{0}
Here, we study the asymptotic behavior of solutions to system \eqref{hkd}. As before, one can prove the following crucial lemma.

	\begin{lem} Let $\{x_i\}_{i=1,\dots,N}$ be a solution to system \eqref{hkd} with continuous initial conditions. Then,
		for each $v\in \RR^{d}$ and $T\geq 0$, we have that 
		\begin{equation}\label{scalpr2}
		\min_{j=1,\dots,N}\min_{s\in[T-\bar\tau,T]}\langle x_{j}(s),v\rangle\leq \langle x_{i}(t),v\rangle\leq \max_{j=1,\dots,N}\max_{s\in[T-\bar\tau,T]}\langle x_{j}(s),v\rangle,
		\end{equation}for all  $t\geq T-\bar\tau$ and $i=1,\dots,N$. 
	\end{lem}
	\begin{proof}
		First of all, we note that the inequalities in the statement are satisfied for every $t\in [T-\bar\tau,T]$.
		\\Now, fix $T\geq 0, $ a vector $v\in \RR^{d}$ and a positive constant $\epsilon.$  Define  the constant $M_{T}$ and the set 
		$K^{\epsilon}$ as in the proof of Lemma \ref{L1}.
 Then, denoted as before $S^{\epsilon}:=\sup K^{\epsilon},$
		it holds that $S^{\epsilon}>T$. \\We claim that $S^{\epsilon}=+\infty$. Indeed, suppose by contradiction that $S^{\epsilon}<+\infty$. Note that by definition of $S^{\epsilon}$ it turns out that \begin{equation}\label{maxd}
		\max_{i=1,\dots,N}\langle x_{i}(t),v\rangle<M_{T}+\epsilon,\quad \forall t\in (T,S^{\epsilon}),
		\end{equation}
		and \begin{equation}\label{tepsd}
		\lim_{t\to S^{\epsilon-}}\max_{i=1,\dots,N}\langle x_{i}(t),v\rangle=M_{T}+\epsilon.
		\end{equation}
		For all $i=1,\dots,N$ and $t\in (T,S^{\epsilon})$, we compute
		$$\begin{array}{l}
			\displaystyle{\frac{d}{dt}\langle x_{i}(t),v\rangle=\frac{1}{h(t)}\sum_{j:j\neq i}\int_{t-\tau_{2}(t)}^{t-\tau_{1}(t)}\alpha(t-s)a_{ij}(t;s)\langle x_{j}(s)-x_{i}(t),v\rangle ds}\\
			\displaystyle{\hspace{2cm}=\frac{1}{N-1}\frac{1}{h(t)}\sum_{j:j\neq i}\int_{t-\tau_{2}(t)}^{t-\tau_{1}(t)}\alpha(t-s)\psi(x_{i}(t), x_{j}(s))(\langle x_{j}(s),v\rangle-\langle x_{i}(t),v\rangle) ds.}
		\end{array}$$
		Notice that, being $t\in (T,S^{\epsilon})$, then  $t-\tau_{2}(t),t-\tau_{1}(t)\in (T-\bar\tau, S^{\epsilon})$ and
		\begin{equation}\label{t-taud}
		\langle x_{j}(s),v\rangle< M_{T}+\epsilon,\quad \forall s\in [t-\tau_{2}(t),t-\tau_{1}(t)],\,\forall j=1, \dots, N.
		\end{equation}
		Moreover, \eqref{maxd} implies that $$\langle x_{i}(t),v\rangle<M_{T}+\epsilon,$$
		so that $$M_{T}+\epsilon-\langle x_{i}(t),v\rangle\geq 0.$$
		Combining this last fact with \eqref{t-taud} and by recalling of \eqref{h(t)}, we can write 
		$$\begin{array}{l}
		\vspace{0.2cm}\displaystyle{\frac{d}{dt}\langle x_{i}(t),v\rangle\leq \frac{1}{N-1}\frac{1}{h(t)}\sum_{j:j\neq i}\int_{t-\tau_{2}(t)}^{t-\tau_{1}(t)}\alpha(t-s)\psi(x_{i}(t),x_{j}(s))(M_{T}+\epsilon-\langle x_{i}(t),v\rangle)ds}\\
		\vspace{0.2cm}\displaystyle{\hspace{2cm}\leq \frac{K}{N-1}\frac{1}{h(t)}(M_{T}+\epsilon-\langle x_{i}(t),v\rangle)\sum_{j:j\neq i}\int_{t-\tau_{2}(t)}^{t-\tau_{1}(t)}\alpha(t-s)ds}\\
		\vspace{0.2cm}\displaystyle{\hspace{2cm}=K\frac{1}{h(t)}(M_{T}+\epsilon-\langle x_{i}(t),v\rangle)\int_{t-\tau_{2}(t)}^{t-\tau_{1}(t)}\alpha(t-s)ds}\\
		\vspace{0.3cm}\displaystyle{\hspace{2cm}=K(M_{T}+\epsilon-\langle x_{i}(t),v\rangle),}
		\end{array}$$
		for all $t\in (T, S^{\epsilon}).$
		Then,  the Gronwall's Lemma allows us to conclude the proof of the second inequality arguing analogously to the proof of Lemma \ref{L1}.
Also the proof of the first inequality is obtained similarly with respect to the pointwise time delay case. We omit the details.

\end{proof}

As before, one can define the quantities $D_n,$ $n\in \N_0,$ and prove the analogous, for solutions to the model with distributed time delay \eqref{hkd}, of the lemmas in Section \ref{Prel}.
Then, the following exponential convergence to consensus holds.

\begin{thm}\label{consd}
 Assume that $\psi:\RR^d\times\RR^d\rightarrow\RR$ is a positive, bounded, continuous function. Then, every solution $\{x_{i}\}_{i=1,\dots,N}$ to \eqref{hkd}, with continuous initial conditions $x_{i}^{0}:[-\bar\tau,0]\rightarrow\RR^{d},$ satisfies the exponential decay estimate 
$$
d(t)\leq \left(\max_{i,j=1,\dots,N}\,\,\max_{r,s\in [-\bar\tau,0]}\lvert x_{i}(r)-x_{j}(s)\rvert\right)e^{-\gamma(t-2\bar{\tau})},\quad \forall  t\geq 0,
$$
for a suitable positive constant $\gamma$ independent of $N.$	
\end{thm}
\begin{oss}
	Let us note that Theorem \ref{consd}  holds true, in particular,  for weights $a_{ij}$ of the form
	$$a_{ij}(t;s):=\frac{1}{N-1}\tilde\psi( \lvert x_{i}(t)-x_{j}(s)\rvert), \quad\forall t>0,\, \forall i,j=1,\dots,N,$$
	where $\tilde\psi:\RR\rightarrow\RR$ is a positive, bounded and continuous function. In this case, one can bound from below the influence function $\tilde\psi$ as explained in Remark \ref{caso-partic}.
\end{oss}

The related PDE model is now:
\begin{equation}\label{kinetic_distr}
\begin{array}{l}
\displaystyle{\partial_t \mu_t+ \mbox{\rm div\,} \left(\frac 1 {h(t)}\int_{t-\tau_2(t)}^{t-\tau_1(t)}\alpha (t-s)F[\mu_{s}] ds\,\mu_t\right )=0, \quad t>0,} 
\\
\displaystyle{\mu_s=g_s, \quad x\in \RR^d, \quad  s\in[-\bar\tau,0],}
\end{array}
\end{equation}
where the velocity field $F$ is given by
\begin{equation}
\label{Fd}
F[\mu_{s}](x)=\int_{\RR^d} \psi (x, y)(y-x)\,d\mu_{s}(y),
\end{equation}
and $g_s \in \mathcal{C}([-\bar{\tau},0];\mathcal{M}(\RR^d))$. For  the continuum model, as in \cite{P}, we  assume that the delay functions are bounded from below, namely there exists a strictly positive constant $\tau^*>0$ such that
\[
\tau_2(t)>\tau_1(t)\geq \tau^*, \quad \forall \,t\geq 0.
\]
Moreover, as before, we assume that the potential $\psi(\cdot,\cdot)$ in \eqref{Fd} is also Lipschitz continuous with respect to the two arguments.
\begin{defn}
Let $T>0$. We say that $\mu_t\in \mathcal{C}([0,T);\mathcal{M}(\RR^d))$ is a measure-valued solution to \eqref{kinetic_distr} on the time interval $[0,T)$ if for all $\varphi\in \mathcal{C}^\infty_c(\RR^d\times [0,T))$ we have:
$$
\int_0^T \int_{\RR^d} \left( \partial_t\varphi+\frac 1 {h(t)}\int_{t-\tau_2(t)}^{t-\tau_1(t)}\alpha (t-s) F[\mu_{s}](x) ds \cdot \nabla_x\varphi \right) d\mu_t(x) dt +\int_{\RR^d} \varphi (x,0)dg_0(x) =0.
$$
\end{defn}

Since the consensus result for the particle model \eqref{hkd} holds without any upper bounds on the time delays $\tau_1(\cdot), \tau_2(\cdot)$, one can improve the consensus theorem for the PDE model \eqref{kinetic_distr} of \cite{P}. Indeed, in \cite{P}, where the author concentrates in the case $\tau_1(t)\equiv 0,$ the consensus estimate is obtained under a smallness condition on the time delay. 
The proof is analogous, then we omit it.

\begin{thm}
Let $\mu_t\in \mathcal{C}([0,T];\mathcal{P}_1(\RR^d))$ be a measure-valued solution to \eqref{kinetic_distr} with compactly supported initial datum $g_s\in \mathcal{C}([-\bar\tau,0];\mathcal{P}_1(\RR^d))$ and let $F$ as in \eqref{Fd}.
Then, there exists a constant $C>0$  such that 
$$
d_X(\mu_t)\leq \left( \max_{s\in[-\bar\tau,0]} d_X(g_s) \right) e^{-Ct},
$$
for all $t\ge 0.$
\end{thm}

\end{document}